\documentclass[a4paper,12pt]{article}
\usepackage[utf8]{inputenc}

\usepackage{boxedminipage}
\usepackage{amsfonts}
\usepackage{amsmath} 
\usepackage{amssymb}
\usepackage{graphicx}
\usepackage{amsthm}
\usepackage{t1enc}
\usepackage{subfig}

\newtheorem{theorem}{Theorem}[section]
\newtheorem{lemma}[theorem]{Lemma}

\newtheorem{corollary}[theorem]{Corollary}
\newtheorem{definition}[theorem]{Definition}

\newtheorem{fact}[theorem]{Fact}
\newtheorem{question}{Question}
\newtheorem*{theorem*}{Theorem}

\newcommand{\forceP}{\mathbb{P}}
\newcommand{\forceQ}{\mathbb{Q}}
\newcommand{\forceR}{\mathbb{R}}

\newcommand{\ZFC}{\mathsf{ZFC}}

\newcommand{\ZFP}{\mathsf{ZF}^-}

\newcommand{\PFA}{\mathsf{PFA}}
\newcommand{\BPFA}{\mathsf{BPFA}}

\newcommand{\CH}{\mathsf{CH}}

\newcommand{\MA}{\mathsf{MA}_{\omega_1}}
\newcommand{\T}{\mathsf{T}}
\newcommand{\PD}{\mathsf{PD}}
\newcommand{\TC}{\mathsf{T}_{\vec{C}}}
\newcommand{\NS}{\hbox{NS}_{\omega_1}}

\def\undertilde#1{\mathord{\vtop{\ialign{##\crcr
$\hfil\displaystyle{#1}\hfil$\crcr\noalign{\kern1.5pt\nointerlineskip}
$\hfil\tilde{}\hfil$\crcr\noalign{\kern1.5pt}}}}}

\title{Forcing Axioms and the Uniformization-Property.}
\author{ Stefan Hoffelner\footnote{WWU M\"unster. The author thanks F. Schlutzenberg and R. Schindler for several discussions on the topic. }}

\begin{document}

\maketitle

\begin{abstract}
We show that there are models of $\MA$ where the $\bf{\Sigma^1_3}$-uniformization property holds. Further we show that $``\BPFA$+ $\aleph_1$ is not inaccessible to reals$"$ outright implies that the $\bf{\Sigma^1_3}$-uniformization property is true. 
\end{abstract}

\section{Introduction}

The question of finding nicely definable 
choice functions for a  definable family of sets is an old and well-studied subject in descriptive set theory.
Recall that for an $A \subset 2^{\omega} \times 2^{\omega}$, we say that $f$ is a uniformization (or a uniformizing function) of
$A$ if there is a function $f: 2^{\omega} \rightarrow 2^{\omega}$, 
$dom(f)= pr_1(A)$ and the graph of $f$ is a subset of $A$.

\begin{definition}
 We say that a pointclass $\Gamma$ has the uniformization
 property iff every element of $\Gamma$ admits a uniformization 
 in $\Gamma$.
\end{definition}

It is a classical result due to M. Kondo that lightface $\Pi^1_{1}$-sets do have the 
uniformization property, this also yields the uniformization property for $\Sigma^1_{2}$-sets. This is all $\ZFC$ can prove about the uniformization property of projective sets. In the constructible universe $L$, for every $n \ge 3$, $\Sigma^1_{n}$ does have the uniformization property which follows from the existence of a good wellorder by an old result of Addison (see \cite{Addison}). Recall that a $\Delta^1_n$-definable wellorder $<$ of the reals is a \emph{good} $\Delta^1_n$-wellorder if $<$ is of ordertype $\omega_1$ and the relation $<_{I} \subset (\omega^{\omega})^2$ defined via
\[ x <_I y \Leftrightarrow \{(x)_n \, : \, n \in \omega\}= \{z \, : \, z < y\} \]
where $(x)_n$ is some fixed recursive partition of $x$ into $\omega$-many reals, is a $\Delta^1_n$-definable. It is easy to check that the canonical wellorder of the reals in $L$ is a good $\Delta^1_2$-wellorder so the $\Sigma^1_n$-uniformization property follows.

On the other hand, the axiom of projective determinacy ($\PD$) draws a very different picture.  Due to a well-known result of Moschovakis (see \cite{Kechris} 39.9), $\PD$ implies that $\Pi^1_{2n+1}$ and $\Sigma^1_{2n+2}$-sets have the uniformization property for $n >1$.
The connection of $\PD$ with forcing axioms is established via core model induction. Under the assumption of the proper forcing axiom, Schimmerling, Steelt and Woodin showed  that $\PD$ is true (in fact much more is true, see \cite{Steel}), thus under $\PFA$ the $\Pi^1_{2n+1}$-uniformization holds for $n>1$. 
As the uniformization property for one pointclass rules out the uniformization property of the dual pointclass, the behaviour of sets of reals in $L$ and under $\PFA$ contradict each other.

The main goal of this paper is to investigate the situation under the weaker forcing axioms $\MA$ and $\BPFA$. We show that for $n=3$ the situation is orthogonal. Indeed both are compatible with $\Sigma^1_3$-uniformization hence the failure of $\Pi^1_{3}$-uniformization. If we add the anti-large cardinal assumption $\omega_1=\omega_1^L$, then the theory $\BPFA$ plus “$\omega_1=\omega_1^L"$ implies the $\Sigma^1_3$-uniformization property outright.

\begin{theorem*}
Starting with $L$ as our ground model, there is a generic extension of $L$ in which $\MA$ and the $\Sigma^1_3$-uniformization property holds true.
\end{theorem*}

In the case of $\BPFA$, the additional assumption that $\omega_1$ is not inaccessible to reals in fact yields a strict implication rather than a consistency result.
\begin{theorem*}
Assume $\BPFA$ and $\omega_1=\omega_1^{L[r]}$ for some real $r$. Then every $\Sigma^1_3$ set in the plane can be uniformized by a $\Sigma^1_3(r)$-definable function. In particular, $\BPFA$ and $\omega_1=\omega_1^L$ implies the $\Sigma^1_3$-uniformization property.
\end{theorem*} 

We end the introduction with a short description of how the article is organized. Section 1.1 introduces the forcings we will use, section 1.2 and 1.3 are concerned with producing a suitable ground model which is akin to a similar model used in \cite{Ho2}, over which we will start the iteration which shall eventually produce a model for the first main theorem, which is proved in 1.4.

The second section is devoted to the proof of the second main theorem; its proofs are very similar to the according ones in section 1. 

\subsection{Notation}
The notation we use will be mostly standard, we hope. We write $\forceP=(\forceP_{\alpha} \, : \, \alpha < \gamma)$ for a forcing iteration of length $\gamma$ with initial segments $\forceP_{\alpha}$. The $\alpha$-th factor of the iteration will be denoted with $\forceP(\alpha)$. When we feel that there is no danger of confusion we tend to drop the dot on $\forceP(\alpha)$, even though $\forceP(\alpha)$ is in fact a $\forceP_{\alpha}$-name of a partial order. 

If $T$ is a tree then we often confuse $T$ with the partial order which comes from the tree order. When forcing with this partial order $T$ we will generically add a cofinal branch through the tree $T$.

We also sometimes write $V[\forceP]\models \varphi$ to indicate that for every $\forceP$-generic filter $G$ over $V$, $V[G] \models \varphi$.

\subsection{The forcings which are used}
The forcings which we will use in the construction are all well-known. We nevertheless briefly introduce them and their main properties. 

\begin{definition}
 For a stationary $R \subset \omega_1$ the club-shooting forcing for $R$, denoted by $\forceP_R$ consists
 of conditions $p$ which are countable functions from $\alpha+1 <\omega_1$ to $R$ whose image is a closed set. $\forceP_R$ is ordered by end-extension.
 \end{definition}
The club shooting forcing $\forceP_R$ is the paradigmatic example for an $R$-\emph{proper forcing}, where we say that $\forceP$ is $R$-proper if and only if for every condition $p \in \forceP$, every $\theta > 2^{| \forceP|}$ (we will utilize the common jargon and say in that situation that $\theta$ is sufficiently large) and every countable $M \prec H(\theta)$ such that $M \cap \omega_1 \in R$ and $p, \forceP \in M$, there is a $q<p$ which is $(M, \forceP)$-generic; and a condition $q \in \forceP$ is said to be $(M,\forceP)$-generic if $q \Vdash ``\dot{G} \cap M$ is an $M$-generic filter$"$, for $\dot{G}$ the canonical name for the generic filter.  See also \cite{Goldstern}. 
\begin{lemma}
 Let $R\subset \omega_1$ be stationary, co-stationary. Then the club-shooting forcing $\forceP_R$ generically adds a club through the stationary set $R \subset \omega_1$. Additionally $\forceP_R$ is $R$-proper, $\omega$-distributive and
 hence $\omega_1$-preserving. Moreover $R$ and all its stationary subsets remain stationary in the generic extension. 
\end{lemma}
\begin{proof}
We shall just show the $\omega$-distributivity of $\forceP_R$, the rest can be found in \cite{Goldstern}, Fact 3.5, 3.6 and Theorem 3.7. Let $p \in \forceP_R$ and $\dot{x}$ be such that $p \Vdash \dot{x} \in 2^{\omega}$. Without loss of generality we assume that $\dot{x}$ is a nice name for a real, i.e. given by an $\omega$-sequence of $\forceP_R$-maximal antichains. We shall find a real $x$ in the ground model and a condition $q < p$ such that $q \Vdash \dot{x}=x$. For this, fix $\theta > 2^{|\forceP|}$ and a countable elementary submodel $M \prec H(\theta)$ which contains $\forceP$, $\dot{x}$ and $p$ as elements and which additionally satisfies that $M \cap \omega_1  \in R$. Note that we can always assume that such an $M$ exists by the stationarity of $R$. We recursively construct a descending sequence $(p_n)_{n \in \omega} \subset M$ of conditions below $p=p_0$ such that every $p_n$ decides the value of $\dot{x}(n)$ and such that both sequences $(dom(p_n))_{n \in \omega}$ and $(max \, range(p_n))_{n \in \omega}$  converge to $M \cap \omega_1$.
We let $x(n) \in 2$ be the value of $\dot{x}$ as forced by $p_n$, and let $x= (x(n))_{n \in \omega} \in 2^{\omega} \cap V$.

Let $q'=\bigcup_{n \in \omega} p_n \subset (M \cap \omega_1)$. We set $q:= q' \cup \{ ((M \cap \omega_1), (M \cap \omega_1)) \}$, which is a function from $(M \cap \omega_1)+1$ to $R$ with closed image, and hence a condition in $\forceP_R$ which forces that $\dot{x} =x$ as desired.

\end{proof}

We will choose a family of $R_{\beta}$'s so that we can shoot an arbitrary pattern of clubs through its elements such that this pattern can be read off
from the stationarity of the $R_{\beta}$'s in the generic extension.
For that it is crucial to recall that for stationary, co-stationary  $R \subset \omega_1$, $R$-proper posets can be iterated with countable support and always yield an $R$-proper forcing again. This is proved exactly as in the well-known case for plain proper forcings (see \cite{Goldstern}, Theorem 3.9 and the subsequent discussion).
\begin{fact}
Let $R \subset \omega_1$ be stationary, co-stationary. Assume that $(\forceP_{\alpha} \, : \, \alpha< \gamma)$ is a countable support iteration of length $\gamma$, let $\forceP_{\gamma}$ denote the resulting partial order and assume also that at every stage $\alpha$, $\forceP_{\alpha} \Vdash  \dot{\forceP}({\alpha})$ is $R$-proper. Then $\forceP_{\gamma}$ is $R$-proper.
\end{fact}

Once we decide to shoot a club through a stationary, co-stationary subset of $\omega_1$, this club will belong to all $\omega_1$-preserving outer models. This hands us a robust method of coding arbitrary information into a suitably chosen sequence of sets
which has been used several times already (see e.g. \cite{SyVera}).
\begin{lemma}\label{coding with stationary sets}
 Let $(R_{ \alpha} \, : \, \alpha < \omega_1)$ be a partition of $\omega_1$ into $\aleph_1$-many stationary sets, let  $r \in 2^{\omega_1}$ be arbitrary, and let $\forceP$ be a countable support iteration $(\forceP_{\alpha} \, : \, \alpha < \omega_1)$, inductively defined via \[\forceP(\alpha) := \dot{\forceP}_{\omega_1 \backslash R_{2 \cdot \alpha}} \text{ if } r(\alpha)=1 \] and
 \[\forceP({\alpha}) := \dot{\forceP}_{\omega_1 \backslash R_{(2 \cdot\alpha) +1}} \text{ if } r(\alpha)=0.\]
 Then in the resulting generic extension $V[\forceP]$, we have that $\forall \alpha < \omega_1:$ \[ r(\alpha)=1 \text{ if and only if }
 R_{2 \cdot \alpha} \text{  is nonstationary, }\]  and \[ r_{\alpha}=0 \text{ iff } R_{(2 \cdot \alpha)+1} \text{ is nonstationary.} \]
\end{lemma}

\begin{proof}
Assume first without loss of generality that $r(0)=1$, then the first factor of the iteration will be $\forceP_{\omega_1 \backslash R_0}$ and is $\omega_1 \backslash R_0$-proper. Now note that $\omega_1 \backslash R_0 \supset R_1$, so $\forceP_{\omega_1 \backslash R_0} $  is $R_1$-proper as well. As the $R_{\alpha}$'s form a partition, $\forceP({\alpha})$ for $\alpha \ge 1$ will always contain $R_1$, hence the iteration is $R_1$-proper, hence $\omega_1$ preserving. 

Now let $\alpha < \omega_1$ be arbitrary and assume that $r(\alpha)=1$. Then by definition of the iteration
 we must have shot a club through the complement of $R_{2 \alpha}$, thus it is nonstationary in $V[{\forceP}]$ as claimed. 
 
On the other hand, if $R_{2 \alpha}$ is nonstationary in $V[{\forceP}]$, then we assume for a contradiction that we did not use $\forceP_{\omega_1 \backslash R_{ 2 \cdot \alpha}}$ in the iteration $\forceP$.
Note that for $\beta \ne 2 \cdot \alpha$, every forcing of the form $\forceP_{\omega_1 \backslash R_{\beta}}$ is $R_{2 \cdot \alpha}$-proper as $\forceP_{\omega_1 \backslash R_{\beta}}$ is $\omega_1 \backslash R_{\beta}$-proper and $R_{2\cdot \alpha} \subset \omega_1 \backslash R_{\beta}$.
 Hence the iteration $\forceP$ will  be $R_{2 \cdot \alpha}$-proper, thus the stationarity of $R_{2 \cdot \alpha}$ is preserved. But this is a contradiction.

 \end{proof}

The second forcing we use is the almost disjoint coding forcing due to R. Jensen and R. Solovay. We will identify subsets of $\omega$ with their characteristic function and will use the word reals for elements of $2^{\omega}$ and subsets of $\omega$ respectively.
Let $D=\{d_{\alpha} \, \: \, \alpha < \aleph_1 \}$ be a family of almost disjoint subsets of $\omega$, i.e. a family such that if $r, s \in D$ then 
$r \cap s$ is finite. Let $X\subset  \kappa$ for $\kappa \le 2^{\aleph_0}$ be a set of ordinals. Then there 
is a ccc forcing, the almost disjoint coding $\mathbb{A}_D(X)$ which adds 
a new real $x$ which codes $X$ relative to the family $D$ in the following way
$$\alpha \in X \text{ if and only if } x \cap d_{\alpha} \text{ is finite.}$$
\begin{definition}
 The almost disjoint coding $\mathbb{A}_D(X)$ relative to an almost disjoint family $D$ consists of
 conditions $(r, R) \in \omega^{<\omega} \times D^{<\omega}$ and
 $(s,S) < (r,R)$ holds if and only if
 \begin{enumerate}
  \item $r \subset s$ and $R \subset S$.
  \item If $\alpha \in X$ and $d_{\alpha} \in R$ then $r \cap d_{\alpha} = s \cap d_{\alpha}$.
 \end{enumerate}
\end{definition}
For the rest of this paper we let $D \in L$ be the definable almost disjoint family of reals one obtains when recursively adding the $<_L$-least real to the family which is almost disjoint from all the previously picked reals. 
Whenever we use almost disjoint coding forcing, we assume that we code relative to this fixed almost disjoint family $D$.

The last two forcings we briefly discuss are Jech's forcing for adding a Suslin tree with countable conditions and, given a Suslin tree $T$, the associated forcing which adds a cofinal branch through $T$. 
Recall that a set theoretic tree $(T, <)$ is a Suslin tree if it is a normal tree of height $\omega_1$
and has no uncountable antichain. As a result, forcing with a Suslin tree $S$, where conditions are just nodes in $S$, and which we always denote with $S$ again, is a ccc forcing of size $\aleph_1$. 
Jech's forcing to generically add a Suslin tree is defined as follows.

\begin{definition}
 Let $\forceP_J$ be the forcing whose conditions are
 countable, normal trees ordered by end-extension, i.e. $T_1 < T_2$ if and only
 if $\exists \alpha < \text{height}(T_1) \, T_2= \{ t \upharpoonright \alpha \, : \, t \in T_1 \}$
\end{definition}
It is wellknown that $\forceP_J$ is $\sigma$-closed and
adds a Suslin tree. In fact more is true, the generically added tree $T$ has 
the additional property that for any Suslin tree $S$ in the ground model
$S \times T$ will be a Suslin tree in $V[G]$. This can be used to obtain a robust coding method (see also \cite{Ho} for more applications)
\begin{lemma}
 Let $V$ be a universe and let $S \in V$ be a Suslin tree. If $\forceP_J$ is 
 Jech's forcing for adding a Suslin tree and if $T$ is the generic tree
 then $$V[\forceP_{J}][T] \models  S \text{ is Suslin.}$$
\end{lemma}

\begin{proof}
Let $\dot{T}$ be the $\forceP_J$-name for the generic Suslin tree. We claim that $\forceP_J \ast \dot{T}$ has a dense subset which is $\sigma$-closed. As $\sigma$-closed forcings will always preserve ground model Suslin trees, this is sufficient. To see why the claim is true consider the following set:
$$\{ (p, \check{q}) \, : \, p \in \forceP_J \land height(p)= \alpha+1  \land  \check{q} \text{ is a node of $p$ of level } \alpha \}.$$
It is easy to check that this set is dense and $\sigma$-closed in $\forceP_J \ast \dot{T}$.

\end{proof}

A similar observation shows that a we can add an $\omega_1$-sequence of
such Suslin trees with a countably supported iteration.

\begin{lemma}\label{ManySuslinTrees}
 Let $S$ be a Suslin tree in $V$ and let $\forceP$ be a countably supported
 product of length $\omega_1$ of forcings $\forceP_J$ with $G$ its generic filter. Then in
 $V[G]$ there is an $\omega_1$-sequence of Suslin trees $\vec{T}=(T_{\alpha} \, : \, \alpha \in \omega_1)$ such
that for any finite $e \subset \omega$
the tree $S \times \prod_{i \in e} T_i$ will be a Suslin tree in $V[G]$.
\end{lemma}

These sequences of Suslin trees will be used for coding in our proof and get a name.
\begin{definition}
 Let $\vec{T} = (T_{\alpha} \, : \, \alpha < \kappa)$ be a sequence of Suslin trees. We say that the sequence is an 
 independent family of Suslin trees if for every finite set $e= \{e_0, e_1,...,e_n\} \subset \kappa$ the product $T_{e_0} \times T_{e_1} \times \cdot \cdot \cdot \times T_{e_n}$ 
 is a Suslin tree again.
\end{definition}
\subsection{The ground model $W$}

We have to first create a suitable ground model $W$ over which the actual iteration will take place. The construction is inspired by \cite{Ho2}, where a very similar universe is used as well. $W$ will be a generic extension of $L$, satisfying $\CH$ and has the crucial property that in $W$ there is an $\omega_2$-sequence $\vec{S}$ of independent Suslin trees which is $\Sigma_1(\omega_1)$-definable over $H(\omega_2)^W$. The sequence $\vec{S}$ will enable a coding method which is independent of the surrounding universe, a feature we will exploit to a great extent in the upcoming.

To form $W$, we start with G\"odels constructible universe $L$ as our 
ground model. Recall that $L$ comes equipped with a $\Sigma_1$-definable, global well-order $<_L$ of its elements.
We first fix an appropriate sequence of stationary, co-stationary subsets of $\omega_1$ using Jensen's $\diamondsuit$-sequence.
\begin{fact}
In $L$ there is a sequence $(a_{\alpha} \, : \, \alpha < \omega_1)$ of countable subsets of $\omega_1$
such that any set $A \subset \omega_1$ is guessed stationarily often by the $a_{\alpha}$'s, i.e.
$\{ \alpha < \omega_1 \, : \, a_{\alpha}= A \cap \alpha \}$ is a stationary and co-stationary subset of $\omega_1$. The sequence $(a_{\alpha} \, : \, \alpha < \omega_1)$ can be defined in a $\Sigma_1$ way over the structure $L_{\omega_1}$.
\end{fact}
\begin{proof}
We shall only prove the claim about the $\Sigma_1$-definability and follow Jensen's original construction of the $\diamondsuit$-sequence. We define a sequence of pairs $(a_{\alpha}, c_{\alpha})$ by induction on $\alpha$. If $\alpha=\beta+1$, then $a_{\alpha}=c_{\alpha}=\alpha$. If $\alpha$ is a limit ordinal, then $(a_{\alpha},c_{\alpha})$ is the $<_L$-least pair such that $c_{\alpha}$ is a closed and unbounded subset of $\alpha$, $a_{\alpha} \subset \alpha$ and such that $a_{\alpha} \cap \eta \ne a_{\eta}$ for every $\eta \in c_{\alpha}$, provided such a pair exists. Otherwise let $a_{\alpha}= c_{\alpha}=\alpha$. It is well-known that the $a_{\alpha}$'s defined this way form a $\diamondsuit$-sequence. We let $\phi(\alpha,x)$ denote the statement: ``$x$ is the $\alpha$-th entry of the $\diamondsuit$-sequence defined as above$"$.

Now it is straightforward to check that $L_{\omega_1}$ is sufficient to correctly compute the sequence $((a_{\alpha},c_{\alpha}) \, : \, \alpha < \omega_1)$ in a $\Sigma_1$-way.
Indeed $L_{\omega_1}$ can correctly compute $L_{\beta}$, for $\beta < \omega_1$ with a $\Sigma_1$-formula. The latter structures, provided $\beta$ is a limit ordinal, are able to define the $<_L$-wellorder up to their respective ordinal height. Thus if the countable $L_{\beta}$, $\beta$ a limit ordinal, contains $((a_{\alpha},c_{\alpha} \, : \, \alpha < \gamma)$, for some $\gamma < \beta$, then $L_{\beta}$ will correctly compute $(a_{\gamma},c_{\gamma})$ as $<_L$ and being closed and unbounded in some $\alpha<\beta$ are absolute notions between $L_{\beta}$ and $L$. Consequentially, being $a_{\alpha}$ is  $\Sigma_1(\alpha)$-definable over $L_{\omega_1}$
\begin{align*}
x= a_{\alpha} \Leftrightarrow \exists \beta (& \beta \text{ is a limit ordinal and }
L_{\beta} \models \phi(\alpha,x)
\end{align*}
and $x \in \{ a_{\alpha} \, : \, \alpha < \omega_1\}$ if and only if $\exists \alpha (x=a_{\alpha})$, which gives the claim.
\end{proof}
 
The $\diamondsuit$-sequence can be used to produce an easily definable sequence of $L$-stationary, co-stationary subsets of $\omega_1$: we list the elements of $L_{\omega_2}$ in an $\omega_2$ sequence $(r_{\alpha} \, : \, \alpha < \omega_2)$. 
We fix \[R:= \{\alpha < \omega_1 \, : \, a_{\alpha}= r_0 \cap \alpha \} \] and note that $R$ is stationary and co-stationary. 

Then we define for every $\beta < \omega_2$, $\beta \ne 0$
a stationary, co-stationary set in the following way:
\[R'_{\beta} := \{ \alpha < \omega_1 \, : \, a_{\alpha}= {r}_{\beta} \cap \alpha \}\]
and 
\[R_{\beta} := \{ \alpha < \omega_1 \, : \, a_{\alpha}= {r}_{\beta} \cap \alpha \} \backslash R.\]
 It is clear that $\forall \alpha \ne \beta (R_{\alpha} \cap R_{\beta} \in \hbox{NS}_{\omega_1})$ and that $R_{\beta} \cap R=\emptyset$. 
To avoid writing $\beta \ne 0$ all the time we re-index and confuse $(R_{\beta} \, : \, \beta < \omega_2, \beta \ne 0)$ with $(R_{\beta} \, : \, \beta < \omega_2)$.
We derive the following standard result concerning the definability of the $R_{\beta}$'s:
\begin{lemma}\label{computationofRbetas}
For any $\beta < \omega_2$, membership in $R_{\beta}$ is uniformly $\Sigma_1(\omega_1)$-definable over the model $L_{\omega_2}$, i.e. there is a $\Sigma_1$-formula with $\omega_1$ as a parameter $\psi(v_0,v_1, \omega_1)$ such that for every $\beta < \omega_2$,
$(\alpha \in R_{\beta} \Leftrightarrow L_{\omega_2} \models \psi(\alpha, \beta,\omega_1))$. 
\end{lemma}
\begin{proof}
First we note that there is a $\Sigma_1(\omega_1)$-formula $\theta'(\eta, x,\omega_1)$ for which 
$L_{\omega_2} \models \theta'(\eta,x,\omega_1)$ is true if and only if ``$x$ is the $\eta$-th subset of $\omega_1$ in $<_L$, the canonical $L$-wellorder$"$. It follows that there is a $\Sigma_1(\omega_1)$-formula $\theta(\eta,\zeta,x,\omega_1)$ for which $L_{\omega_2} \models \theta(\eta, \zeta,x,\omega_1)$ is true if and only if ``$x$ equals ${r}_{\eta} \cap \zeta"$. Further recall that in the proof of the last lemma we found already a $\Sigma_1$-formula, let us denote it with $\varphi(\xi,y)$, such that $L_{\omega_1} \models \varphi(\xi,y)$ holds if and only if ``$y$ is the $\xi$-th element of the canonical $\diamondsuit$-sequence$"$.

Then membership in $R'_{\beta}$ can be expressed using the following formula:
\begin{align*}
\alpha \in R'_{\beta} \Leftrightarrow L_{\omega_2} \models   \exists x( \varphi(\alpha,x) \land \theta(\beta, \alpha,x,\omega_1))
\end{align*}
Note here that actually every $\aleph_1$-sized $L_{\gamma}$, for $\gamma$ a limit ordinal, which models (a sufficiently big fragment of) $\ZFP$ and contains $\alpha$ and $\beta$ is sufficient to witness membership of $\alpha$ in $R'_{\beta}$ using the formula $\exists x( \varphi(\alpha,x) \land \theta(\beta, \alpha,x,\omega_1))$. In particular, not being an element of $R$ can be written in a $\Sigma_1(\omega_1)$ fashion as well:
\begin{align*}
\alpha \notin R \Leftrightarrow L_{\omega_2} \models \exists \gamma >\omega_1 (L_{\gamma} \models  \ZFP \land \exists x( \varphi(\alpha,x) \land \lnot \theta(0, \alpha,x,\omega_1))
\end{align*}

It follows that membership in $R_{\beta}$ allows this representation:
\begin{align*}
\alpha \in R_{\beta} \Leftrightarrow L_{\omega_2} \models &  \exists x( \varphi(\alpha,x) \land \theta(\beta, \alpha,x,\omega_1)) \land \\&
\qquad \exists \gamma > \omega_1( L_{\gamma} \models ``\ZFP \land \ \\& \qquad \qquad \qquad \qquad \quad
 \exists y \, (\varphi(\eta, y) \land \lnot \theta (0,\alpha,y,\omega_1)" ))
\end{align*}
Note that the last formula is $\Sigma_1(\omega_1)$, thus we found our desired $\psi(v_0,v_1,\omega_1)$.
\end{proof}

We proceed with defining the universe $W$.
First we generically add $\aleph_2$-many Suslin trees using of Jech's Forcing $ \forceP_J$. We let 
\[\forceQ^0 := \prod_{\beta \in \omega_2} \forceP_J \] using countable support. This is a $\sigma$-closed, hence proper notion of forcing. We denote the generic filter of $\forceQ^0$ with $\vec{S}=(S_{\alpha} \, : \, \alpha < \omega_2)$ and note that by Lemma \ref{ManySuslinTrees} $\vec{S}$ is independent.  We fix a definable bijection between $[\omega_1]^{\omega}$ and $\omega_1$ and identify the trees in $(S_{\alpha }\, : \, \alpha < \omega_1)$ with their images under this bijection, so the trees will always be subsets of $\omega_1$ from now on.

In a second step we code the trees from $\vec{S}$ into the sequence of $L$-stationary subsets $\vec{R}$ we produced earlier, using the method introduced in Lemma \ref{coding with stationary sets}. It is important to note, that the forcing we are about to define does preserve Suslin trees, a fact we will show later.
The forcing used in the second step will be denoted by $\mathbb{Q}^1$ and will itself be a countable support iteration of length $\omega_1$ whose factors are itself countably supported iterations of length $\omega_1$. Fix first a definable bijection $h \in L_{\omega_3}$ between $\omega_1 \times \omega_2$ and $\omega_2$ and write $\vec{R}$ from now on in ordertype $\omega_1 \cdot \omega_2$ making implicit use of $h$, so we assume that $\vec{R}= (R_{\alpha} \, : \, \alpha < \omega_1 \cdot \omega_2)$. 

$\forceQ^1$ is defined via induction in the ground model $L[\forceQ^0]$. Assume we are at stage $\alpha < \omega_2$ and we have already created the iteration $\forceQ^1_{\alpha}$ up to stage $\alpha$. We work with $L[\forceQ^0][\forceQ^1_{\alpha}]$ as our ground model
 and consider the Suslin tree $S_{\alpha} \subset \omega_1$. 
 We define the forcing we want to use at stage $\alpha$, denoted by $\forceQ^1 (\alpha)$, as the countable support iteration which codes the characteristic function of $S_{\alpha}$ into the $\alpha$-th $\omega_1$-block of the $R_{\beta}$'s just as in Lemma \ref{coding with stationary sets}. So $\forceQ^1 (\alpha)= \bigstar_{\gamma < \omega_1} \forceR_{\alpha}(\gamma)$ is again a countable support iteration in $L[\forceQ^0][\forceQ^1_{\alpha}]$, defined inductively via
\[ \forall \gamma < \omega_1 \,(\forceR_{\alpha}(\gamma):= \dot{\forceP}_{\omega_1 \backslash R_{\omega_1 \cdot \alpha + 2 \gamma +1}}) \text{ if } S_{\alpha} (\gamma) =0 \]
and
\[ \forall \gamma < \omega_1 \, (\forceR_{\alpha}(\gamma):= \dot{\forceP}_{\omega_1 \backslash R_{\omega_1 \cdot \alpha + 2 \gamma}}) \text{ if } S_{\alpha} (\gamma) =1. \]

Recall that we let $R$ be a stationary, co-stationary subset of $\omega_1$ which is disjoint from all the $R_{\alpha}$'s which are used.  We let $\mathbb{Q}^1$ be the countably supported iteration, $$\mathbb{Q}^1:=\bigstar_{\alpha< \omega_2} \forceQ^1_{\alpha}$$ which is an $R$-proper forcing. We shall see later that $\forceQ^1$ in fact is $\omega$-distributive, hence the iteration $\forceQ^1$ is in fact a countably supported product.
This way we can turn the generically added sequence of Suslin trees $\vec{S}$ into a definable sequence of Suslin trees.
Indeed, if we work in $L[\vec{S}\ast G]$, where $\vec{S} \ast G$ is $\forceQ^0 \ast \mathbb{Q}^1$-generic over $L$, then, as seen in Lemma \ref{coding with stationary sets} 
\begin{align*}
\forall \alpha< \omega_2, \gamma < \omega_1 (&\gamma \in S_{\alpha} \Leftrightarrow R_{\omega_1 \cdot \alpha + 2 \cdot \gamma} \text{ is not stationary and} \\ &
\gamma \notin S_{\alpha} \Leftrightarrow  R_{\omega_1 \cdot \alpha + 2 \cdot \gamma +1} \text{ is not stationary})
\end{align*}
Note here that the above formula can be used to make every $S_{\alpha}$ $\Sigma_1(\omega_1,\alpha)$ definable over $L[\vec{S} \ast G]$ as is shown with the next lemma.

\begin{lemma}\label{definabilityofvecS}
There is a $\Sigma_1(\omega_1)$-formula $\Phi(v_0,v_1,\omega_1)$ such that every $\gamma < \omega_2$,

\[ L[\vec{S} \ast G] \models \forall \beta < \omega_1 (\Phi (\beta, \gamma,\omega_1) \Leftrightarrow \beta \in S_{\gamma}) \]
 
Thus initial segments  of the sequence $\vec{S} $ are uniformly  $\Sigma_1(\omega_1)$-definable over $L[\vec{S} \ast G]$.
\end{lemma}
\begin{proof}
Let $\gamma < \omega_2$.
We claim that already $\aleph_1$-sized, transitive models of $\ZFP$ which contain a club through the complement of exactly one element of every pair $\{(R_{\alpha}, R_{\alpha+1}) \, : \, \alpha < \omega_1 \cdot \gamma\}$ are sufficient to compute correctly $\vec{S} \upharpoonright \gamma$ via the following $\Sigma_1(\gamma, \omega_1)$-formula: 

\begin{align*}
\Psi(X,\gamma, \omega_1) \equiv  \exists M (&M \text{ transitive } \land M \models \ZFP \land \omega_1,\gamma \in M \land
\\& M \models \forall \beta< \omega_1 \cdot \gamma (\text{either  $R_{2\beta}$ or $R_{2\beta+1}$ is nonstationary) } \land \\& 
M \models X \text{ is an $\omega_1 \cdot \gamma$-sequence $(X_{\alpha})_{\alpha < \omega_1 \cdot \gamma}$ of subsets of $\omega_1$} \land\\&
M \models  \forall \alpha, \delta (\delta \in X_{\alpha} \Leftrightarrow R_{\omega_1 \cdot \alpha + 2 \cdot \delta} \text{ is not stationary and} \\& \qquad \qquad \quad \,
\delta \notin X_{\alpha} \Leftrightarrow  R_{\omega_1 \cdot \alpha + 2 \cdot \delta +1} \text{ is not stationary})
\end{align*}
We want to show that 
\begin{align*}
X=\vec{S} \upharpoonright \gamma \text{ if and only if } \Psi(X,\gamma,\omega_1) \text{ is true in }L[\vec{S} \ast G].
\end{align*}

For the backwards direction, we assume that $M$ is a model and $X \in M$ is a set, as on the right hand side of the above. We shall show that indeed $X=\vec{S}\upharpoonright \gamma$.  As $M$ is transitive and a model of $\ZFP$ it will compute every $R_{\beta}$, $\beta < \omega_1 \cdot \gamma$ correctly by Lemma \ref{computationofRbetas}. As being nonstationary is a $\Sigma_1(\omega_1)$-statement, and hence upwards absolute, we conclude that if $M$ believes to see a pattern written into (its versions of) the $R_{\beta}$'s, this pattern is exactly the same as is seen by the real world $L[\vec{S} \ast G]$. But we know already that in $L[\vec{S} \ast G]$, the sequence $\vec{S}$ is written into the $R_{\beta}$'s, thus $X=\vec{S} \upharpoonright \gamma$ follows.

On the other hand, if $X=\vec{S} \upharpoonright \gamma$, then
\begin{align*}
L[\vec{S} \ast G] \models \forall \beta< \omega_1 \cdot \gamma (\text{either  $R_{2\beta}$ or $R_{2\beta+1}$ is nonstationary) } \\ 
L[\vec{S} \ast G]  \models X \text{ is an $\omega_1 \cdot \gamma$-sequence $(X_{\alpha})_{\alpha < \omega_1 \cdot \gamma}$ of subsets of $\omega_1$}
\end{align*}
and
\begin{align*}
L[\vec{S} \ast G] \models  \forall \alpha, \delta < \omega_1 (\delta  \in X_{\alpha} \Leftrightarrow & R_{\omega_1 \cdot \alpha + 2 \cdot \delta } \text{ is not stationary and} \\ 
\delta  \notin X_{\alpha} \Leftrightarrow&  R_{\omega_1 \cdot \alpha + 2 \cdot \delta +1} \text{ is not stationary})
\end{align*}
By reflection, there is an $\aleph_1$-sized, transitive model $M$ which models the assertions above, which gives the direction from left to right.

\end{proof}

Let us set \[W:= L[\forceQ^0\ast \forceQ^1 ]\] which will serve as our ground model for a second iteration of length $\omega_2$. We shall need the following
well-known result (see \cite{Abraham}, Theorem 2.10 pp.20, or see \cite{SyVera} Lemma 12).
\begin{fact}
Let $R\subset \omega_1$ be stationary and co-stationary.
Assume $\CH$ and let $(\forceP_{\alpha} \, : \, \alpha \le \delta)$ be a
countable support iteration of $R$-proper posets such
that for every $\alpha \le \delta$
\[ \forceP_{\alpha} \Vdash |\forceP(\alpha)| = \aleph_1. \]
Then $\forceP_{\delta}$ satisfies the $\aleph_2$-cc.
\end{fact}

\begin{lemma}
$W$ is an $\omega$-distributive generic extension of $L$ which also satisfies $\aleph_2^L=\aleph_2^W$.
\end{lemma}
\begin{proof}
The second assertion follows immediately from the last Fact. The first assertion holds by the following argument which already will look familiar. First note that as $\forceQ^0$ does not add any reals it is sufficient to show that $\forceQ^1$ is $\omega$-distributive.
Let $p \in \forceQ^1$ be a condition and assume that $p \Vdash \dot{r} \in 2^{\omega}$. We shall find a stronger $q < p$ and a real $r$ in the ground model such that $q \Vdash \check{r}=\dot{r}$. Let $M \prec H(\omega_3)$ be a countable elementary submodel which contains $p, \forceQ^1$ and $\dot{r}$ and such that $M \cap \omega_1 \in R$, where $R$ is our fixed stationary set from above. Inside $M$ we recursively construct a decreasing sequence $p_n$ of conditions in $\forceQ^1$, such that for every $n$ in $\omega,$ $p_n \in M$, $p_n$ decides $\dot{r}(n)$ and for every $\alpha$ in the support of $p_n$, the sequence $sup_{n \in \omega} max( p_n(\alpha))$ converges towards $M \cap \omega_1$ which is in $R$. Now, $q':= \bigcup_{n \in \omega} p_n$ and for every $\alpha< \omega_1$ such that $q'(\alpha)\ne 1$ (where 1 is the weakest condition of the forcing),  in other words for every $\alpha$ in the support of $q'$ we define $q(\alpha):= q'(\alpha) \cup \{((M \cap \omega_1), (M \cap \omega_1))\}$ and $q(\alpha)=1$ otherwise. Then $q=(q(\alpha))_{\alpha < \omega_1}$ is a condition in $\forceQ^1$, as can be readily verified and $q \Vdash \dot{r} = \check{r}$, as desired.

The second assertion follows immediately from the well-known theorem that under $\CH$, a countable support iteration of $R$-proper forcings of size $\aleph_1$,
\end{proof}
Note that by the last lemma, the second forcing $\forceQ^1$ which is an the countably supported iteration of the appropriate club shooting forcings is in fact just a countably supported product of its factors, i.e. at every stage of the iteration we can force with the according $\forceP_R$, as computed in $L$.

Our goal is to use $\vec{S}$ for coding again. For this it is essential, that the sequence remains independent in $W$. To see this we shall argue that forcing with $\mathbb{Q}^1$ over $L[\forceQ^0]$ preserves Suslin trees. 
The following  line of reasoning is similar to \cite{Ho}.
Recall that for a forcing $\forceP$, $\theta$ sufficiently large and regular and $M \prec H(\theta)$, a condition $q \in \forceP$ is $(M,\forceP)$-generic iff for every maximal antichain $A \subset \forceP$, $A \in M$, it is true that $ A \cap M$ is predense below $q$. In the following we will write $T_{\eta}$ to denote the $\eta$-th level of the tree $T$ and $T \upharpoonright \eta$ to denote the set of nodes of $T$ of height $< \eta$.
The key fact is the following (see \cite{Miyamoto2} for the case where $\forceP$ is proper)
\begin{lemma}\label{preservation of Suslin trees}
 Let $T$ be a Suslin tree, $R \subset \omega_1$ stationary and $\forceP$ an $R$-proper
 poset. Let $\theta$ be a sufficiently large cardinal.
 Then the following are equivalent:
 \begin{enumerate}
  \item $\Vdash_{\forceP} T$ is Suslin
 
  \item if $M \prec H_{\theta}$ is countable, $\eta = M \cap \omega_1 \in R$, and $\forceP$ and $T$ are in $M$,
  further if $p \in \forceP \cap M$, then there is a condition $q<p$ such that 
  for every condition $t \in T_{\eta}$, 
  $(q,t)$ is $(M, \forceP \times T)$-generic.
 \end{enumerate}

\end{lemma}

\begin{proof}
For the direction from left to right note first that $\Vdash_{\forceP} T$ is Suslin implies $\Vdash_{\forceP} T$ is ccc, and in particular it is true that for any countable elementary submodel $N[\dot{G}_{\forceP}] \prec H(\theta)^{V[\dot{G}_{\forceP}]}$,  $\Vdash_{\forceP} \forall t \in T (t$ is $(N[\dot{G}_{\forceP}],T)$-generic). Now if $M \prec H(\theta)$ and $M \cap \omega_1 = \eta \in R$ and $\forceP,T \in M$ and $p \in \forceP \cap M$ then there is a $q<p$ such that $q$ is $(M,\forceP)$-generic. So $q \Vdash \forall t \in T (t$ is $(M[\dot{G}_{\forceP}], T)$-generic, and this in particular implies that $(q,t)$ is $(M, \forceP \times T)$-generic for all $t \in T_{\eta}$. 

For the direction from right to left assume that $\Vdash \dot{A} \subset T$ is a maximal antichain. Let $B=\{(x,s) \in \forceP \times T \, : \, x \Vdash_{\forceP} \check{s} \in \dot{A} \}$, then $B$ is a predense subset in $\forceP \times T$. Let $\theta$ be a sufficiently large regular cardinal and let $M \prec H(\theta)$ be countable such that $M \cap \omega_1=\eta \in R$ and $\forceP, B,p,T \in M$. By our assumption there is a $q <_{\forceP} p$  such that $\forall t \in T_{\eta} ((q,t)$ is $(M, \forceP \times T)$-generic). So $B \cap M$ is predense below $(q,t)$ for every $t \in T_{\eta}$, which yields that $q \Vdash_{\forceP} \forall t \in T_{\eta} \exists s<_{T} t(s \in \dot{A})$ and hence $q \Vdash \dot{A} \subset T \upharpoonright \eta$, so $\Vdash_{\forceP} T$ is Suslin.
\end{proof}
We strongly believe that in a similar way, one can show that Theorem 1.3 of \cite{Miyamoto2} holds true if we replace proper by $R$-proper for $R \subset \omega_1$ a stationary subset, i.e. that a countable support iteration of $R$-proper forcings which preserve Suslin trees results in a forcing which preserves Suslin trees. Instead of proving this we will take a more direct route though and prove that $\forceQ^1$ preserves Suslin trees by hand.

\begin{lemma}\label{omegadistributive}
Let $R \subset \omega_1$ be stationary, co-stationary, then the club shooting forcing $\forceP_R$ preserves Suslin trees.
\end{lemma}

\begin{proof}
Let $T$ be an arbitrary Suslin tree from the ground model $V$.
Because of Lemma \ref{preservation of Suslin trees}, it is enough to show that for any regular and sufficiently large
 $\theta$, every $M \prec H_{\theta}$ with $M \cap \omega_1 = \eta \in R$, and every
 $p \in \forceP_R \cap M$ there is a $q<p$ such that for every
 $t \in T_{\eta}$, $(q,t)$ is $(M,(\forceP_R \times T))$-generic.
 Note first that, as $T$ is Suslin, every node $t \in T_{\eta}$ is an
 $(M,T)$-generic condition. Further, as forcing with a Suslin tree
 is $\omega$-distributive, $(H(\omega_1))^{M[G]} = (H(\omega_1))^M$ for every $T$-generic filter $G$ over $V$. As $\forceP_R  \subset H(\omega_1)$, we obtain that the set $M[G] \cap \forceP_R$ is independent of the choice of the generic filter $G$ and equals $M \cap \forceP_R$. Likewise $M[G] \cap \omega_1= M \cap \omega_1$, for every $T$-generic filter.
 
 Next we note that for a countable $M$ and a $V$-generic filter $G \subset T$, 
 the model $M[G]$ is (up to isomorphism) uniquely determined by the $t \in T_{\eta}$, such that $t \in G$ and $\eta= M \cap \omega_1$. This is clear as we can transitively collapse $M[G]$ to obtain a structure of the form $\bar{M}[t]$, where $\bar{M}$ is the image of $M$ under the collapse map and $t \in T_{\eta}$ is the unique node in $T$ to which $G$ is sent to by the collapse map. 
So for a countable $M \prec H(\theta)$, and $\eta= M \cap \omega_1$, we write $M[t]$ for the unique model of the form $M[G]$, for $ G$ $T$-generic over $V$ and $t \in G \cap T_{\eta}$.
With an argument almost identical to the one used in the proof of Lemma 2.2 it is not hard to see that if $M\prec H(\theta)$ is such
 that $M \cap \omega_1 \in R$ then an $\omega$-length descending sequence
 of $\forceP_R$-conditions in $M$ whose domains converge to $M \cap \omega_1$
 has a lower bound as $M \cap \omega_1 \in R$.
 
 We construct an $\omega$-sequence of elements of $\forceP_R$ which has a lower bound
 which will be the desired condition $q$ such that for every $t \in T_{\eta}$, $(q,t)$ is $(M, \forceP_R \times T)$-generic. 
 We list the nodes on $T_{\eta}$, $(t_i \, : \, i \in \omega)$ and
 consider the according generic extensions $M[t_i]$.
 In every $M[t_i]$ we list the $\forceP_R$-dense subsets of $M[t_i]$,
 $(D^{t_i}_n \, : \, n \in \omega)$, write 
 the so listed dense subsets of $M[t_i]$ as an $\omega \times \omega$-matrix and enumerate
 this matrix in an $\omega$-length sequence of dense sets $(D_i \, : \, i \in \omega)$.
 If $p=p_0 \in \forceP_R \cap M$ is arbitrary we can find, using the fact that $\forall i \, (\forceP_R \cap M[t_i] = M \cap \forceP_R$), an $\omega$-length, descending
 sequence of conditions below $p_0$ in $\forceP_R \cap M$, $(p_i \, : \, i \in \omega)$
 such that $p_{i+1} \in M \cap \forceP_R$ is in $D_i$.
 By the usual density argument we can conclude that the domain of the conditions $p_i$ converge to $M[t_i] \cap \omega_1=M \cap \omega_1$.
 Then the $p_i$'s have a lower bound $q=p_{\omega} \in \forceP_R$, namely $q= \bigcup_{i \in \omega} p_i \cup \{(\eta, \eta)\}$  and $(t, q)$ is an
 $(M, T \times \forceP_R)$-generic condition for every $t \in T_{\eta}$ as any $t \in T_{\eta}$ is $(M,T)$-generic
 and every such $t$ forces that $q$ is $(M[T], \forceP_R)$-generic; moreover $q < p$ as 
 desired.
 \end{proof}
The above quickly generalizes to $\forceQ^1$.
\begin{lemma}
The second forcing $\forceQ^1$ we use to produce $W=L[\forceQ^0 \ast \forceQ^1]$ preserves Suslin trees.
\end{lemma}
\begin{proof}
We work over the ground model $L[\forceQ^0]$. Let $T$ be a Suslin tree from $L[\forceQ^0]$. As by definition of $\forceQ^1$, there is an $R \subset \omega_1$ for which $\forceQ^1$ is $R$-proper, we, again, shall show that for any regular and sufficiently large
 $\theta$, every $M \prec H_{\theta}$ with $M \cap \omega_1 = \eta \in R$, and every
 $p \in \forceQ^1 \cap M$ there is a $q<p$ such that for every
 $t \in T_{\eta}$, $(q,t)$ is $(M,(\forceQ^1 \times T))$-generic.
 
We note that by the remark following the proof of Lemma \ref{omegadistributive}, $\forceQ^1$ is equivalent to the countably supported product of its factors. As a consequence of this and the fact that forcing with a Suslin tree is $\omega$-distributive, we obtain that for any $M \prec H(\theta)$,
and any $G$ which is $T$-generic over $L[\forceQ^0]$, $H(\omega_1)^{M[G]}= H(\omega_1)^M$ and as $\forceQ^1 \subset H(\omega_1)^{L[\forceQ^0]}$, it must hold that $\forceQ^1 \cap M[G] = \forceQ^1 \cap M$, and is thus independent of $G$.
 
The proof now follows closely the one of the last Lemma, and uses its notions. 
We start with an arbitrary $p \in \forceQ^1 \cap M$ and need to find a $q<p$ such that $(q,t)$ is $(M, \forceQ^1 \times T)$-generic for every $t \in T_{\eta}$.
We first list all $T_{\eta}$-nodes $(t_i)_{i \in \omega}$. In every $M[t_i]$, $t \in T_{\eta}$, we ist the dense sets $\{D^{t_i}_n \subset \forceQ^1 \, : \, n \in \omega\}$ and re-list the resulting $\omega \times \omega$ matrix of dense subsets in ordertype $\omega$. Now, starting with our condition $p:=p_0$ we form a descending sequence of conditions such that for every $n \in \omega$, $p_n \in D_n$. By density the support of $p_n$ will converge to $M \cap \omega_1=\eta$ as $n$ grows, and likewise the maximum of the domain and the maximum of the range of every coordinate of $p_n$ will converge to $\eta$ as well.
Thus $q:= \bigcup_{n \in \omega} p_n \cup \{ (\eta, \eta, \xi) \, : \, \xi < \eta \}$ is a condition in $\forceQ^1$ below $p$. By construction, every $t \in T_{\eta}$ forces that $q$ is $M[t],\forceQ^1)$-generic, and as forcing with $T$ has the ccc every $t \in T_{\eta}$ is trivially $(M,T)$-generic. As a consequence, $(q,t)$ is $M,\forceQ^1 \times T)$-generic for every $t\in T_{\eta}$ as desired.
\end{proof}
To summarize the main results from above
\begin{theorem}
The universe $W=L[\forceQ^0][\forceQ^1]$ is an $\omega$-distributive, $\aleph_2$-preserving generic extension of $L$ and contains $\vec{S}$ which is an independent sequence of Suslin trees. Moreover, for arbitrary $\alpha < \omega_2$, $\vec{S} \upharpoonright \alpha$ is uniformly $\Sigma_1(\alpha,\omega_1)$-definable over $W$.
\end{theorem}

We end with a straightforward lemma which is used later in coding arguments.

\begin{lemma}\label{a.d.coding preserves Suslin trees}
 Let $T$ be a Suslin tree and let $\mathbb{A}_D(X)$ be the almost disjoint coding which codes
 a subset $X$ of $\omega_1$ into a real with the help of an almost disjoint family
 of reals $D$ of size $\aleph_1$. Then $$\mathbb{A}_{D}(X) \Vdash_{} T \text{ is Suslin }$$
 holds.
\end{lemma}
\begin{proof}
 This is clear as $\mathbb{A}_{D}(X)$ has the Knaster property, thus the product $\mathbb{A}_{D}(X) \times T$ is ccc and $T$ must be Suslin in $V[{\mathbb{A}_{D}(X)}]$. 
\end{proof}

\subsection{Coding over $W$ to produce $\tilde{W}$}
We shall use the just created $W$ to start a second iteration which will code a fresh sequence of Suslin trees $\vec{T}$ into the $W$-definable sequence $\vec{S}$ of Suslin trees. 
The reason for this seemingly redundant choice, is that we later want to force $\MA$, which of course will necessarily create a lot of noise on any definable sequence of independent Suslin trees. This noise is a threat to any intentional coding we aim to build in order to have a universe for the $\Sigma^1_3$-uniformization property.

The addition of the second sequence $\vec{T}$ of independent Suslin trees enables us to later use coding with $\vec{T}$ trees along the usual forcing which produces $\MA$ and yet be in full control of the codes we create on the $\vec{T}$-sequence. This effect can not be reproduced when forcing over just $W$, this is why we need to pass from $W$ to a more suitable generic extension.

The universe $\tilde{W}=W[\forceP_0 \ast \forceP_1]$ is a two step generic extension of $W$. The first factor $\forceP_0$ is just an $\omega_2$-length product of ordinary Cohen forcing with finite support
\[\forceP_0:= \prod_{\alpha < \omega_2} \mathbb{C}. \]
As is known due to S. Shelah (see \cite{Jech}, Theorem 28.12, Lemma 28.13) every generically added Cohen real $r$ can be used to define a new Suslin tree $T_r$. To be more precise, if $V$ is our ground model and
\[ (e_{\alpha} \, : \, \alpha < \omega_1) \] is an $\omega_1$ sequence of functions in $V$ which satisfies
\begin{enumerate}
\item for every $\alpha < \omega_1$, $e_{\alpha}$ is an injection from $\alpha$ to $\omega$;
\item and for every $\alpha < \beta < \omega_1$, $e_{\alpha} (\xi)=e_{\beta} (\xi)$ for all but finitely many $\xi < \alpha$,
\end{enumerate}
then
\[ T_r :=\{ r \circ (e_{\alpha} \upharpoonright \beta) \, : \, \alpha,\beta < \omega_1\} \]
 is a Suslin tree, as long as $r$ is a Cohen real over $V$.
 
 Consequentially, if $(c_{\alpha} \, : \, \alpha < \omega_2)$ denotes the sequence of Cohen reals in $W[\forceP_0]$, the sequence of trees 
 \[\vec{T}:=(T_{\alpha} \, : \, \alpha < \omega_2) \]
 is an independent sequence of Suslin trees. Indeed, any finite product of $T_{\alpha}$'s must be a Suslin tree again, as otherwise, if say $T_1$ and $T_2$ are such that $T_1 \times T_2$ is not Suslin, then $\mathbb{C} \times \mathbb{C} \times T_1 \times T_2$ is not a ccc forcing. But rearranging the factors yields that $T_2$ is a Suslin tree in $V[\mathbb{C}][T_1][\mathbb{C}]$, hence $\mathbb{C} \times T_1 \times \mathbb{C} \times T_2$ has the ccc, which is a contradiction.
 
 Our second forcing $\forceP_1$ uses our already created independent sequence $\vec{S}$ to code up the trees from $\vec{T}$ with a method completely analogous to our forcing $\forceQ^1$, we used in the definition of $W$. That is if we fix an arbitrary $\alpha < \omega_2$ with the according tree $T_{\alpha} \subset \omega_1$, 
and for $\beta < \omega_1$, we let
 \[ \forceP^{\beta}_{1,\alpha}= \begin{cases} S_{ \omega_1 \alpha + \beta} & \text{ if $T_{\alpha} (\beta)=1$ }
 \\
 \hbox{SP} (S_{\omega_1 \alpha+\beta}) & \text{ if $T_{\alpha} (\beta) =0$ }
 \end{cases} \]
where $S_{ \omega_1 \alpha + \beta}$ just denotes the forcing which adds an $\omega_1$-branch through $S_{ \omega_1 \alpha + \beta}$, and SP$(S_{ \omega_1 \alpha + \beta})$ is the forcing which specializes $S_{ \omega_1 \alpha + \beta}$.

Then
\[ \forceP_{1,\alpha} = \prod_{\beta < \omega_1} \forceP^{\beta}_{1,\alpha} \] using finite support. Note that for every $\alpha < \omega_2$, $\forceP_{1,\alpha}$ has the ccc.

Finally we let
\[ \forceP^1 := \prod_{\alpha < \omega_2} \forceP_{1,\alpha} \]
again using finite support. The upshot of these manipulations is, that $\vec{T}$ is now a uniformly definable $\omega_2$-sequence of independent Suslin trees which has a second crucial feature, we discuss later and which is not shared by $\vec{S}$.

We let $\tilde{W}= W[\forceP_0][\forceP_1]$.

\begin{lemma}\label{definabilityofvecT}
The universe $\tilde{W}$ is a cardinal preserving generic extension of $L$, whose continuum is $\aleph_2$. Moreover
there is a $\Sigma_1(\omega_1)$-formula $\Xi(v_0,v_1,\omega_1)$ such that every $\gamma < \omega_2$,

\[ \tilde{W} \models \forall \beta < \omega_1 (\Xi (\beta, \gamma,\omega_1) \Leftrightarrow \beta \in T_{\gamma}) \]
 
Thus initial segments  of the sequence $\vec{T} $ are uniformly  $\Sigma_1(\omega_1)$-definable over $\tilde{W}$.
\end{lemma}
\begin{proof}
The proof of the definability claim is almost identical to the proof of Lemma \ref{definabilityofvecS}, so we omit it. The rest is just a summary of the discussions above.
\end{proof}

\subsection{Iteration over $\tilde{W}$}
Having defined $\tilde{W}$ we start an iteration of length $\omega_2$. In the resulting universe $\MA$ and the $\Sigma^1_3$-uniformization property will be true.

We isolate first a class of transitive models which we take advantage of when showing the $\Sigma^1_3$-uniformization property.

The following is inspired by a similar definition in \cite{CV}:
Let $\T$ be the theory consisting of the following sentences:
\begin{itemize}
\item $\ZFP$,
\item $\forall x (|x| \le \aleph_1)$,
\item Every $X \subset \omega_1$ is coded into a real $r_X$ with the help of the canonical almost disjoint family $D$ of reals from $L$.
\item For every real $r$ there is an ordinal $\gamma$ such that the block $(T_{\gamma \cdot \omega + n}\,  : \, n \in \omega)$ codes $r$ via
\begin{align*}
\forall n \in \omega (& n \in r \Leftrightarrow T_{\omega \gamma + n} \text{ has a cofinal branch} \,  \land \\ &n \notin r \Leftrightarrow T_{\omega \gamma + n} \text{ is special})
\end{align*}
\item For every pair of ordinals $(2\beta, 2\beta+1)$ either $R_{2\beta}$ or $R_{2\beta+1}$, where $S_{\beta}$ is our definable $L$-stationary subset of $\omega_1$, is not stationary.
\item For every ordinal $\gamma$, $S_{\gamma}$ has either a cofinal branch or is special.
\item For every ordinal $\gamma$, $T_{\gamma}$ has either a cofinal branch or is special.
\end{itemize} 
It is important to note that $\tilde{W}$ provides the low-complexity definitions to enable the following absoluteness property between  $\T$-models and generic extensions of $\tilde{W}$:
\begin{lemma}
Let $\tilde{W}[G]$ be a generic extension of $\tilde{W}$ which preserves $\omega_1$. Let $M \in \tilde{W}[G]$ be a transitive $\T$-model  of size $\aleph_1$. Then if $r \in M\cap 2^{\omega}$ and
\begin{align*}
M \models \forall n \in \omega ( n \in r &\Leftrightarrow T_{\omega \gamma + n} \text{ has a cofinal branch} \, \land \\
n \notin r & \Leftrightarrow T_{\omega \gamma + n } \text{ is special})
\end{align*} 
then 
\begin{align*}
\tilde{W}[G] \models \forall n \in \omega ( n \in r &\Leftrightarrow T_{\omega \gamma + n} \text{ has a cofinal branch} \, \land \\
n \notin r & \Leftrightarrow T_{\omega \gamma + n } \text{ is special})
\end{align*} 
Consequentially, the reals which are seen to be coded by an $\aleph_1$-sized, transitive $\T$-model into $\vec{T}$ are coded into $\vec{T}$ in the real world.
\end{lemma}
\begin{proof}
We note first that any transitive, $\aleph_1$-sized $\T$-model $M$ will define the sequence $\vec{T}$ up to $M \cap \omega_2$ correctly as long as it has all the information available. Indeed if $M$ is such that
\begin{itemize}
\item $M \models \ZFP$,
\item $M$ is transitive,
\item $\aleph_1 \in M$, and
\item for every pair of ordinals $(2\beta, 2\beta+1)$ either $R_{2\beta}$ or $R_{2\beta+1}$, where $S_{\beta}$ is our definable $L$-stationary subset of $\omega_1$, is not stationary.
\item For every ordinal $\gamma$, $S_{\gamma}$ has either a cofinal branch or is special.
\end{itemize}
Then, using the $\Sigma_1(\omega_1)$-formula $\Xi(v_0,v_1,\omega_1)$ from Lemma \ref{definabilityofvecT}, $M$ will compute $\vec{T} \upharpoonright (M \cap \omega_2)$ correctly. As $``$having a cofinal branch$"$, and $``$being special$"$ are both $\Sigma_1(\omega_1)$-properties, the pattern $M$ thinks it sees on its version of the $\vec{T}$-sequence must be the patterns which we find in $\tilde{W}[G]$, as otherwise $\omega_1$ would have been collapsed.

\end{proof}
As an immediate consequence, $\T$-models form a well-defined stratification of subsets of $H(\omega_2)$ of $\tilde{W}[G]$, the latter being an $\omega_1$-preserving generic extension of $\tilde{W}$.

\begin{lemma}
Work in a universe $\tilde{W}[G] \supset \tilde{W}$ which satisfies that $\omega_1^{\tilde{W}[G]} =\omega_1^{\tilde{W}}$.
Let $M \in \tilde{W}[G]$ be a transitive model of size $\aleph_1$ which satisfies $\T$. Then $M$ is uniquely determined by $\omega_2 \cap M$.
\end{lemma}
\begin{proof}
Every $\aleph_1$-sized, transitive $\T$-model $M$ is, by the definition of $\T$, completely determined by the reals it contains. As every real in $M$ is coded by an $\omega$-block of Suslin trees $(T_{\omega \gamma +n} \, : \, n \in \omega)$, every $\T$-model is determined by the set of codes which are written on its block of Suslin trees $\vec{T}$. In $\tilde{W}[G]$, there is exactly one such sequence and it can be read of by transitive, $\aleph_1$-sized models $\T$. Thus, for every $\alpha < \omega_2$, there is at most one transitive $\T$-model $M$ such that $M \cap \omega_2= \alpha$.
\end{proof}

Now we return to defining our iteration with $\tilde{W}$ as the ground model.
Two types of forcings are used, one for creating a hierarchy of $\T$-models tied together with a localization forcing, and the second one for working towards $\MA$. We will use finite support and we use a bookkeeping function $F: \omega_2 \rightarrow (\omega_2)^2$ to organize the iteration. The choice of $F$ does not really matter, we will only assume that $F$ is surjective and that for every $(\alpha,\beta) \in (\omega_2)^2$, $F^{-1} (\alpha,\beta)$ is unbounded in $\omega_2$.
The iteration will be defined inductively, suppose we are at a stage $\alpha < \omega_2$ and we have already defined the iteration $\forceP_{\alpha}$ up to $\alpha$ and a generic filter $G_{\alpha}$ for $\forceP_{\alpha}$. We shall define the next forcing $\dot{\forceQ}_{\alpha}$ we want to use next.
\begin{itemize}
\item Suppose $\alpha=\beta+1$ for some ordinal $\beta$. Then we let $\dot{\forceQ}_{\alpha}$ be a two step iteration $\forceR_0 \ast \forceR_1$ where $\forceR_0$ is itself an iteration of length $\omega_1$, whereas $\forceR_1$ will serve as a coding forcing. 
The forcing $\forceR_0$ will produce a new $\T$-model. We force in $\omega_1$-many steps a generic extension $\tilde{W}[G_{\alpha+1}]$ of $\tilde{W}G_{\alpha}]$ such that there is a new, transitive model $M$ of $\T$ of size $\aleph_1$ and height $\beta_{\alpha} < \omega_2$. There are many ways how to achieve this and the definition is not dependent of which one we pick. We let $H(\alpha)$ be a $\forceR_0$-generic filter over $\tilde{W}[G_{\alpha}]$ and define $\forceR_1$ now. Let $X_{\alpha} \subset \omega_1$ be a set which codes $M$. First we note that we can rewrite $X_{\alpha}$ into a new set $Y_{\alpha} \subset \omega_1$ which has the additional property that its information can be read off already by suitable countable transitive models as we will see later.

The considerations which will yield the desired set $Y_{\alpha}$ are as follows. Note first that the set of countable $X \prec L_{\omega_2} [M]$ which contain $M$ give rise to a club 
\[C:= \{ \beta < \omega_1 \, : \, \exists X \prec L_{\omega_2}[M] \land M \in X \land \omega_1 \cap X =\beta \}. \] 
We let $\{c_{\beta} \,: \, \beta < \omega_1\}$ be the enumeration of $C$.
We define $Y_{\alpha} \subset \omega_1$ such that the odd entries code the $X_{\alpha}$ and the even entries $E(Y_{\alpha})$ of $Y_{\alpha}$ satisfiy:
\begin{enumerate}
\item $E(Y_{\alpha}) \cap \omega$ codes a well-ordering of type $c_0$.
\item $E(Y_{\alpha}) \cap [\omega, c_0) = \emptyset$.
\item For all $\beta$, $E(Y_{\alpha}) \cap [c_{\beta}, c_{\beta} + \omega)$ codes a well-ordering of type $c_{\beta+1}$.
\item For all $\beta$, $E(Y_{\alpha}) \cap [c_{\beta}+\omega, c_{\beta+1})= \emptyset$.
\end{enumerate}

In the next step we use almost disjoint forcing $\mathbb{A}_D(Y_{\alpha})$ relative to the $<_L$-least almost disjoint family of reals $D$ to code the set $Y_{\alpha}$ into one real $r_{\alpha}$. The almost disjoint coding forcing shall be our second forcing $\forceR_1$.

\item Suppose that $\alpha$ is a limit ordinal, then force towards obtaining a model for $\MA$. We pick with the help of the bookkeeping function $F$ the according ccc poset $\forceP=F(\alpha)$ of size $\aleph_1$ in a diagonal way. To be more precise, we let $(\beta,\gamma)$ be such that $F(\alpha)=(\beta,\gamma)$ and pick the $\gamma$-th name (in some fixed well-order of $H(\omega_2)$)
of a forcing $\dot{\forceQ}$ in $\tilde{W}^{\forceP_{\beta}}$ which is ccc in $\tilde{W}[G_{\alpha}]$ and which has the additional property that there is a $\delta < \omega_2$ such that if we write $\tilde{W}$ as $W[G^0][G^1]$, then $\dot{\forceQ}$ is actually in $W[G^0_{\delta}][G^1_{\delta}] [G_{\alpha}]$.
If this is the case, then we force with $\dot{\forceQ}_{\alpha}:=F(\alpha)$, but only if $\dot{\forceQ}_{\alpha}$ satisfies that \[\tilde{W}[G_{\alpha}] \models \dot{\forceQ}_{\alpha}^{G_{\alpha}} \Vdash``\text{There are $\aleph_2$-many trees in $\vec{T}$ which are still Suslin}".\]
Otherwise we force with the trivial forcing.
\end{itemize}
This ends the definition of our iteration. We add as a remark, that at every intermediate stage $\alpha< \omega_2$ of the iteration, by definition, there will still be $\aleph_2$-many Suslin trees from our definable sequence $\vec{T}$ which are still Suslin in the intermediate model $L[G_{\alpha}]$, where $G_{\alpha}$ as always denotes the generic for the iteration cut at $\alpha$. Thus in the iteration, we never run out of trees, so the length of it will be $\omega_2$.

\subsubsection{Properties of $W_1$}
After $\omega_2$-many steps we arrive at a model $\tilde{W}[G_{\omega_2}]=:W_1$ which will satisfy the $\Sigma^1_3$-uniformization property and $\MA$ as we shall show now.
The following is straightforward from the definition of the iteration.
\begin{lemma}
 $H(\omega_2)^{W_1} \models \T$. In particular, every subset of $\omega_1$ is coded by a real and every real is coded into an $\omega$-block of our definable sequence of Suslin trees $\vec{T}$.
\end{lemma}

The upshot of forcing $\MA$ in the diagonal fashion utilized in the definition of $\forceP_{\omega_2}$ are the two following lemmas
\begin{lemma}
Let $\alpha < \omega_2$ and let $\forceP_{\alpha}$ be the initial segment of $\forceP_{\omega_2}$. Then
\[\tilde{W} \models |\forceP_{\alpha}| = \aleph_1 \]
\end{lemma}
\begin{proof}
The proof is by induction on $\alpha < \omega_2$.
For limit stages $\alpha$, the assertion follows from the fact that we use finite support in our iteration. For successor stages $\alpha+1$, we note that $\forceP_{\alpha+1}=\forceP_{\alpha} \ast \dot{\forceQ}_{\alpha}$ and that by the way we defined $\forceP_{\omega_2}$, there must be a stage $\beta_{\alpha} < \omega_2$ such that $\forceP_{\alpha} \in W[G^0_{\beta_{\alpha}}][G^1_{\beta_{\alpha}}]$. Indeed if $\alpha$ is a limit ordinal, then this follows immediately from our definition of the iteration and if $\alpha$ is successor then this follows from the fact that in the process of producing the next $\T$-model, the only forcings which are used are trees from $\vec{T}$ and almost disjoint coding forcings relative to the canonical almost disjoint family of reals. Likewise there is $\beta_{\alpha+1} < \omega_2$ such that $\dot{\forceQ}_{\alpha}$ is in fact in $W[G^0_{\beta_{\alpha+1}}] [G^1_{\beta_{\alpha+1}}]$.
Now we note that $\CH$ is true in $W[G^0_{\beta_{\alpha+1}}] [G^1_{\beta_{\alpha+1}}]$, as $\CH$ is true in $W$ and both forcings are $\aleph_1$-length iterations of ccc forcings over $W$. This implies that $\dot{\forceQ}_{\alpha}$ is essentially an $\aleph_1$-sequence of  countable antichains in $\forceP_{\alpha}$ which is an element of $W[G^0_{\beta_{\alpha+1}}] [G^1_{\beta_{\alpha+1}}]$.
Hence $|\forceP_{\alpha+1}| = \aleph_1$, and the lemma is proved.
\end{proof}

\begin{lemma} In $W_1$, $\MA$ holds.
\end{lemma}
\begin{proof}
Assume that $\forceQ \in W_1=\tilde{W}[G_{\omega_2}]$ is an $\aleph_1$-sized partial order with the ccc and that $D$ is an $\aleph_1$-sized family of dense subsets of $\forceQ$. Then there is a stage $\beta < \omega_2$, such that, if $G_{\beta}$ denotes the generic filter for the intermediate forcing $\forceP_{\beta} \subset \forceP_{\omega_2}$, then $\forceQ$ and $D$ are already elements of $\tilde{W}[G_{\beta}]$. 

By the (proof of the) last lemma, we know that $\forceP_{\beta}$ is an $\aleph_1$-sized forcing as seen from $\tilde{W}$, and there are ordinals $\alpha_{\beta} < \omega_2$, such that
$\forceP_{\beta}$ is in fact an element of $W[G^0_{\alpha_{\beta}} ][G^1_{\alpha_{\beta}}]$.

But then, $\forceP_{\beta} \ast \forceQ$ is an element of $W[G^0_{\alpha_{\beta}} ][G^1_{\alpha_{\beta}}]$, and $W[G^0_{\alpha_{\beta}} ][G^1_{\alpha_{\beta}}] \models ``|\forceP_{\beta} \ast \forceQ| = \aleph_1$ and has the ccc$"$.

By our choice of the bookkeeping function $F$, there will be cofinally many stages $\gamma> \beta$ such that the forcing considered by $F$ at stage $\gamma$ is $\forceQ$. In particular there will be a least such stage, denoted again by $\gamma > \beta$, such that $D \in \tilde{W}[G_{\gamma}]$ and $\forceQ$ is suggested by $F$ at $\gamma$. We claim that at stage $\gamma$,
\[\tilde{W}[G_{\gamma}] \models \forceQ \Vdash ``\text{There are $\aleph_2$-many trees in $\vec{T}$ which are still Suslin}". \]
If the claim is true, then by definition, we must have used $\forceQ$ at stage $\gamma$, and so $\MA$ must be true.

To see why the claim holds, we first let $H$ be $\forceQ$-generic and note that if $\alpha_{\gamma}$ denotes a stage $<\omega_2$ such that $\forceP_{\gamma} \in W[G^0_{\alpha_{\gamma}}] [G^1_{\alpha_{\gamma}}]$, we can factor $G^0_{\omega_2}$ and $G^1_{\omega_2}]$ into $G^i_{\alpha_{\gamma}} \ast G^i_{[\alpha_{\gamma},\omega_2)}$, and hence
\[ W[G^0_{\omega_2}][G^1_{\omega_2}][G_{\gamma}] [H] = W[G^0_{\alpha_{\gamma}}] [G^1_{\alpha_{\gamma}}] [G_{\gamma}][H] [G^0_{[\alpha_{\gamma},\omega_2)} ] [G^1_{[\alpha_{\gamma},\omega_2)} ]. \]
So, in particular, the $\aleph_2$-many Cohen reals added by $G^0_{[\alpha_{\gamma}, \omega_2)}$ will still be Cohen-generic reals when considered over the ground model
$W[G^0_{\alpha_{\gamma}}] [G^1_{\alpha_{\gamma}}] [G_{\gamma}][H]$
As a consequence 
\begin{align*}
W[G^0_{\alpha_{\gamma}}] [G^1_{\alpha_{\gamma}}] [G_{\gamma}] [H] [G^0_{[\alpha_{\gamma},\omega_2)} ] [G^1_{[\alpha_{\gamma},\omega_2)} ] \models &``\text{There are $\aleph_2$-many trees in $\vec{T}$} \\&  \qquad\text{ which are still Suslin.} "
\end{align*}
hence
\begin{align*}
 W[G^0_{\omega_2}][G^1_{\omega_2}][G_{\gamma}] [H] \models &``\text{There are $\aleph_2$-many trees in $\vec{T}$} \\&  \qquad\text{ which are still Suslin.} "
\end{align*}
Thus the claim is true and we must have used $\forceQ$ at stage $\gamma$.

\end{proof}

We know already that transitive $\T$-models of size $\aleph_1$ are uniquely determined by their ordinal height. As we forced with many almost disjoint coding forcings, we obtain a projectively definable set of reals which code $\T$-models. For a real $r$ we let $(r)_1$ and $(r)_2$ denote its recursive splitting into its even and its odd part respectively.
\begin{lemma}\label{realscodeTmodels}
In $W_1$, for every $\gamma < \omega_2$ there is a $\T$-model $M$ of ordinal height $>\gamma$, and there is a real $r_M$ which satisfies that $(r_M)_2 = r_{(M \cap \omega_2)}$, where $r_{(M \cap \omega_2)}$ is a real which codes the ordinal $M \cap \omega_2$, and which additionally satisfies:
\begin{align*}
(\ast) \quad \forall N (|N| = \aleph_0& \, \land N \text{ is transitive and }  \omega_1^N=(\omega_1^{L})^N  \land r_M \in N \\& \rightarrow
N \models  ``L[r_M] \models \text{ $(r_M)_1$ codes a $\T$-model   of height $\alpha((r_M)_2)$,} \\& \qquad \qquad \qquad \qquad \text{ where $\alpha((r_M)_2)$ is the ordinal $< \omega_2^{N}$ } \\&\qquad \qquad \qquad \qquad \text{ decoded from the real $(r_M)_2$.}")
\end{align*}
If $N$ is uncountable and transitive  and $r_M$ is as above, then
\begin{align*}
N \models ``\text{The model decoded out of $r_M$ is a $\T$-model.}"
\end{align*}

\end{lemma}
\begin{proof}
We defined our iteration such that at successor stages, we always produce $\T$-models and subsequently code them into reals. We let $\gamma < \omega_2$ be arbitrary and consider an $\alpha=\beta+1 > \gamma$, i.e. a stage past $\gamma$ where we produce a new $\T$-model $M$ and a real $r$ which codes the set $Y_{\alpha}$, where $Y_{\alpha}$ is as in the definition of the iteration.

If we let $s$ be a real which codes the ordinal height of $M$, i.e. a real which codes a set $X_s \subset \omega_1$ which in turn codes a wellorder of ordertype $M \cap \omega_2$, then we claim that the real $x$ which consists of $r$ on its even, and $s$ on its odd entries is as desired, i.e. $x=r_M$.

To see this, assume that $N$ is countable, transitive, $\omega_1^N=(\omega_1^L)^N$ and $r_M \in N$.
Now first note that if $N$ contains $r_M$, it also must contain $Y_{\alpha} \cap \omega_1^N$, as the latter is coded into $r$ which in turn is coded into $r_M$, and $N$ can use its local almost disjoint family of $L$-reals $D \upharpoonright \omega_1^N=D \upharpoonright (\omega_1^L)^N$ to successfully compute $Y_{\alpha} \cap \omega_1^N$. 
This means in particular, that there is a countable $X \prec L_{\omega_2}[M]$, such that the transitive collapse of $X$, denoted with $\bar{X}$ and $N$ have the same $\omega_1$.
Consequentially, if  we consider $(L_{\omega_1}[r_M])^N$, then $(L_{\omega_1}[r_M])^N$ is a countable initial segment of $L[r_M]$ and is equal to the $(L_{\omega_1}[r_M])^{\bar{X}}$ which is formed inside $\bar{X}$. By elementarity, \[ \bar{X} \models ``L_{\omega_1}[r_M] \models \text{$(r_M)_1$ codes a $\T$-model of height } \alpha((r_M)_2)", \] and so also
\[ N \models ``L_{\omega_1}[r_M] \models (r_M)_1 \text{ codes a $\T$-model of height }\alpha((r_M)_2)" .\]
But this is already the first assertion of the lemma.

To show that for any $\aleph_1$-sized, transitive $N$, $N\models ``$The model decoded out ot $(r_M)_1$ is $M"$, we assume for a contradiction that this is wrong and $N$ is a witness that the assertion is wrong. Then we can pick a countable \[X \prec N\] which contains $r_M$, and its transitive collapse $\bar{X}$ would see that 
\[\bar{X} \models ``L[r_M] \models (r_M)_1 \text{ does not code a $\T$-model}". \] 
Now it is sufficient to note that $\bar{X}$ is a model as in the first assertion; indeed it is countable, transitive, contains $r_M$ and its local $L$ computes $\omega_1$ correctly, as $N$ does so. But this is a contradiction.
\end{proof}

Note that the statement $(\ast)$ is a $\Pi^1_2(r_{M})$-assertion and
as a consequence there is a $\Pi^1_2$-definable set of reals which code, in the sense of $(\ast)$, $\T$-models.

The fact that $\T$-models are stratified by their ordinal height enables the
following definition of a well-order of the reals, in fact of $P(\omega_1)$.
For $x,y \in 2^{\omega}$, let $M_x$ be the least $\T$-model which contains $x$ and likewise define $M_y$. 
If $x \in M$, then we can assign an ordinal $\alpha_x$ to $x$ which is the least ordinal such that $M$ thinks that there is an $\omega$-block of Suslin trees from $\vec{T}$ starting at $\alpha_x$ which codes $x$ i.e it is true that 
\[ n \in x \Leftrightarrow M \models T_{\alpha_x+n} \text{ has a branch. }\]
and
\[n \notin x \Leftrightarrow M \models T_{\alpha_x +n} \text{ is special. }\]
This induces a wellorder of the reals, in fact a wellorder of $P(\omega_1)$. For $x,y \in 2^{\omega}$, we let
\[x < y \Leftrightarrow \alpha_x < \alpha_y.\]

\begin{lemma}\label{uniformization}
In $W_1$, the $\Sigma^1_3$-uniformization property holds.
\end{lemma}
\begin{proof}
Let $\varphi(v_0,v_1)=\exists v_2 \psi(v_0,v_1,v_2)$ be an arbitrary $\Sigma^1_3$-formula with two free variables, let $x \in 2^{\omega}$ and assume that $W_1 \models \exists y\varphi(x,y)$, thus the $x$-section of the set $A \subset 2^{\omega} \times 2^{\omega}$ defined via $\varphi(v_0,v_1)$ is non-empty.

We collect the set of $M$ where $M$ is a $\T$-model, for which there is a real $r_M$ coding it, and which contains the reals $x,y$ and sees that $\exists z \psi (x,y,z)$ is true. Note that by Shoenfield absolutness, this implies that $\varphi(x,y)$ really is true.
As there is a definable wellorder $<$ of the reals, we pick $M$ such that its $<$-least code is minimal among all such reals coding such $\T$-models. With the help of our coding this can be written in a $\Pi^1_2(r_M)$-way:
\begin{align*}
\forall N (r_M \in N \land \omega_1^N=(\omega_1^L)^N \rightarrow
&N \text{ decodes out of } r_M \text{ a } \T \text{-model } M \text{ such that } \\& M \text{ is the $<$-least such model which sees that }\\& \exists z \psi(x,y,z) \text{ is true.})
\end{align*}
The fact that $\T$-models can internally define a wellorder of its reals, can be used to single out a pair of reals $(y,z)$ which $M$ thinks is the least to satisfy $\psi(x,y,z)$.  Then we let $y$ be the first coordinate of this $<$-least pair just defined and set $y$ to be the value of our uniformizing function $f$ at $x$.
Thus the relation $$f(x)=y$$ can be defined as follows:
\begin{align*}
(\heartsuit) \quad \exists r \forall N (r \in N \land \omega_1^N=(\omega_1^L)^N \rightarrow
&N \text{ decodes out of } r \text{ a } \T \text{-model } M \text{ such that } \\& M \text{ is the $<$-least such model which sees that }\\& \exists y\exists z \psi(x,y,z) \text{ is true and which thinks that for all} \\& \text{ pair of reals } (y',z') < (y,z)  \,\varphi(x,y',z') \text{ is wrong.} )
\end{align*}
Note that this is a $\Sigma^1_3$-formula and we take this to be the definition of $f$.

What is left is to show that $f$ defined via $(\heartsuit)$ really defines a uniformizing function for the $\Sigma^1_3$-set defined by $\{ (x,y) \, : \, \varphi(x,y)\}$. We shall show three things, first we show that $f$ is always well-defined; second we show that whenever $f(x)=y$, then $\varphi(x,y)$ is true; and third we show that $dom f= \{x \, : \, \exists y (\varphi(x,y)\}$.

First, we shall show that $f$ is well-defined. Assume for a contradiction that there are $y\ne y'$ such that $f(x)=y$ and $f(x)=y'$ defined via $(\heartsuit)$ is true and let $r_y$ and $r_{y'}$ be the reals which witness the truth of $f(x)=y$ and $f(x)=y'$ in $(\heartsuit)$.
Now, if $H(\theta)$ is sufficiently big, then by lemma \ref{realscodeTmodels} it will decode out of $r_{y}$ and $r_{y'}$ two $\T$-models $M_y$ and $M_{y'}$ which are both $<$-least for thinking that $\exists y \exists z (\psi(x,y,z))$ is true in $M_y$ and $M_{y'}$. So $M_y=M_{y'}$ and hence $y=y'$. Thus $f(x)$ is welldefined.

Second we assume that $f(x)=y$ holds, and let $r$ be the real which witnesses that the property $(\heartsuit)$ is true. Then, by lemma \ref{realscodeTmodels}, if $N$ is some sufficiently big $H(\theta)$, it is clear that $r$ is in fact a code for a true $\T$-model $M$, which in turn must think that $\varphi(x,y)$ is true. But then, by Shoenfield absoluteness, $\varphi(x,y)$ is true so $f(x)=y$ implies that $\varphi(x,y)$ is true.

Third, if $x$ is such that there is a $y$ such that $\varphi(x,y)$ is true, then there is a $\T$-model $M$ which sees that as well.
Surely there will be a bigger $\T$-model $M' \supset M$ which contains $r_M$, where the latter is a real which codes $M$.
As $M'$ is coded by some real $r_{M'}$ we have that $r_{M'}$ must satisfy $(\heartsuit)$, hence $x \in dom f$. This shows that if
$\exists y (\varphi(x,y))$, then $x \in dom f$. The other inclusion is already shown in the second claim, so the third claim is shown.

\end{proof}

The above generalizes to the boldface case immediately.
\begin{corollary}
In $W_1$ for every real $r$, every $\Sigma^1_3(r)$-relation in the plane can be uniformized by a $\Sigma^1_3(r)$-definable function.
\end{corollary}

\section{$\BPFA$ and the $\Sigma^1_3$-uniformization property.}

The results of the last section can be strengthened if we assume $\BPFA$.
We aim to prove the following result.
\begin{theorem}
Assume $\BPFA$ and $\omega_1=\omega_1^L$, then the $\Sigma^1_3$-uniformization property holds.
\end{theorem}
Note that the above is not a consistency result. Its proof makes heavy use of a coding method invented by A. Caicedo and B. Velickovic (see \cite{CV}) as well as an important observation by A. Caicedo and S. D. Friedman that one can easily combine the Coding method with David's trick (see \cite{SyAndres}). As it is somewhat difficult to describe their technique in a short way without leaving out too many important features out, we opt to introduce their coding method briefly. 

\begin{definition}
A $\vec{C}$-sequence, or a ladder system, is a sequence $(C_{\alpha} \, : \, \alpha \in \omega_1, \alpha \text{ a limit ordinal })$, such
that for every $\alpha$, $C_{\alpha} \subset \alpha$ is cofinal and the order type of $C_{\alpha}$ is $\omega$.
\end{definition}
As we always work with $L$ as our ground model, there is a canonical ladder system $\vec{C} \in L$ which can be defined using $L$'s $\Delta^1_2$-definable well-order of the reals. From now on whenever we write $\vec{C}$, we have this canonical ladder system in mind.

For three subsets $x,y,z \subset \omega$ we can consider the oscillation function. First, turn the set $x$ 
into an equivalence relation $\sim_x$, defined on the set $\omega- x$ as follows: for natural numbers in the 
complement of $x$ satisfying $n \le m$, let $n \sim_x m$ if and only if $[n,m] \cap x = \emptyset$.
This enables us to define:
\begin{definition}
For a triple of subset of the natural numbers $(x,y,z)$ list the intervals $(I_n \, :\, n \in k \le \omega)$ of equivalence classes of
$\sim_x$ which have nonempty intersection with both $y$ and $z$. Then, the oscillation map $o(x,y,z):
k \rightarrow 2$ is defined to be the function satisfying

\begin{equation*}
o(x,y,z)(n) = \begin{cases}
0  & \text{ if min}(I_n \cap y) \le \text{min}(I_n \cap z) \\ 1 & \text{ else.}
\end{cases}
\end{equation*}

\end{definition}

Next, we want to define how suitable countable subsets of ordinals can be used to code reals. 
For that suppose that $\omega_1 < \beta < \gamma < \delta$ are fixed limit ordinals, and that 
$N \subset M$ are countable subsets of $\delta$.
Assume further that $\{ \omega_1, \beta, \gamma\} \subset N$ and that for every 
$\eta \in \{ \omega_1, \beta, \gamma\}$, $M \cap \eta$ is a limit ordinal and $N \cap \eta < M \cap \eta$.
We can use $(N,M)$ to code a finite binary string. Namely, let $\bar{M}$ denote the transitive collapse of 
$M$, let $\pi : M \rightarrow \bar{M}$ be the collapsing map and let 
$\alpha_M := \pi(\omega_1)$, $\beta_M := \pi(\beta), \, \gamma_M := \pi(\gamma) \, \delta_M:= \bar{M}$. 
These are all countable limit ordinals.
Furthermore set $\alpha_N:= sup(\pi``(\omega_1 \cap N))$ and let the height $n(N,M)$ of $\alpha_N$ in $\alpha_M$ 
be the natural number defined by

$$n(N,M):= card (\alpha_N \cap C_{\alpha_M}),$$ where $C_{\alpha_M}$ is an element of our previously fixed ladder system. 
As $n(N,M)$ will appear quite often in the following we write shortly $n$ for $n(N,M)$. Note that
as the order type of each $C_{\alpha}$ is $\omega$, and as $N \cap \omega_1$ is bounded below $M \cap \omega_1$,
$n$ is indeed a natural number.
Now, we can assign to the pair $(N,M)$ a triple $(x,y,z)$ of finite subsets of natural numbers as follows:
$$x:= \{ card(\pi(\xi) \cap C_{\beta_M}) \, : \, \xi \in \beta \cap N \}.$$ Note that $x$ again is finite as $\pi" (\beta \cap N)$ is 
bounded in the cofinal in $\beta_M$-set $C_{\beta_M}$, which has ordertype $\omega$. Similarly we define 
$$y:= \{ card(\pi(\xi) \cap C_{\gamma_M}) \, : \, \xi \in \gamma \cap N \}$$ and
$$z:= \{ card(\pi(\xi) \cap C_{\delta_M} \, : \, \xi \in \delta \cap N \}.$$ Again, it is easily 
seen that these sets are finite subsets of the natural numbers.
We can look at the oscillation $o(x \backslash n, y \backslash n, z \backslash n)$ 
and if the oscillation function at these points has a domain bigger or equal to $n$ then we write
\begin{equation*}
s_{\beta, \gamma, \delta} (N,M):= \begin{cases}
o(x \backslash n, y \backslash n, z \backslash n)\upharpoonright n & \text{ if defined } \\ \ast \text{ else.}
\end{cases}
\end{equation*}
We let $s_{\beta, \gamma, \delta} (N,M) \upharpoonright l = \ast$ when $l \geq n$.
Finally we are able to define what it means for a triple of ordinals $(\beta, \gamma, \delta)$ to code a real $r$.

\begin{definition}
For a triple of limit ordinals $(\beta, \gamma, \delta)$, we say that it codes a real $r \in 2^{\omega}$
if there is a continuous increasing sequence $(N_{\xi} \, : \, \xi < \omega_1)$ of countable sets of ordinals , also called a reflecting sequence, whose 
union is $\delta$ and which satisfies that whenever $\xi < \omega_1$ is a limit ordinal then there is a $\nu < \xi$ such that
$$ r = \bigcup_{\nu < \eta < \xi} s_{\beta, \gamma, \delta} (N_{\eta}, N_{\xi}). $$
\end{definition}
The technicalities in the definitions are justified by the fact that $\BPFA$ suffices to introduce witnesses to the codings.

\begin{theorem}[Caicedo-Velickovic]
Assume that $\BPFA$ holds, then:

\begin{enumerate}
\item[$(\dagger)$] Given ordinals $\omega_1 < \beta < \gamma < \delta < \omega_2$ of cofinality $\omega_1$,
then
there is a reflecting, i.e.,  increasing and continuous sequence $(N_{\xi} \, : \, \xi < \omega_1)$ such that $N_{\xi} \in [\delta]^{\omega}$ 
whose union is $\delta$ such that for every limit $\xi < \omega_1$ and every $n \in \omega$ there 
is $\nu < \xi$ and $s_{\xi}^n \in 2^n$ such that 
$$s_{\beta \gamma \delta}(N_{\eta}, N_{\xi}) \upharpoonright n = s_{\xi}^{n}$$ holds for every $\eta$ in the interval $(\nu, \xi)$.
We say then that the triple $(\beta, \gamma, \delta)$ is \textit{stabilized}.
 
\item[$(\ddagger)$] Further if we fix a real $r$
there is a triple of ordinals
$(\beta_r, \gamma_r, \delta_r)$ of size and cofinality $\aleph_1$ and a reflecting sequence 
$(P_{\xi} \, : \, \xi < \omega_1)$, $P_{\xi} \in [\delta_r]^{\omega}$ such that $\bigcup_{\xi < \omega_1} P_{\xi} = \delta_r$ 
and such that for every limit $\xi < \omega_1$ there is a $\nu < \xi$ such that 
$$\bigcup_{\nu < \eta < \xi} s_{\beta_r \gamma_r \delta_r} (P_{\eta}, P_{\xi}) = r.$$ 
We say then that the real $r$ is \textit{coded} by the triple $(\beta_r, \gamma_r, \delta_r)$.
\end{enumerate}
\end{theorem}

The coding induces a hierarchy on $H(\omega_2)$ whose initial segments are $\Sigma_1(\omega_1)$-definable. As we work over $L$, we can easily define ladder system $\vec{C}$ via inductively picking always the $<_L$-least real coding a cofinal set. From now on we exclusively work with this ladder system $\vec{C}$. Recall that we fixed a similarly definable, almost disjoint family of reals $D$.

\begin{definition} Let $\T_{\vec{C}}$ denote the following list of axioms:
\begin{enumerate}
\item $\forall x (|x| \le \aleph_1)$,
\item $\ZFP$,
\item Every subset of $\omega_1$ is coded by a real, relative to the almost disjoint family $D$.
\item Every triple of limit ordinals is stabilized in the sense of $\dagger$ using $\vec{C}$.
\item Every real is determined by a triple of ordinals in the sense of $\ddagger$ using $\vec{C}$.
\end{enumerate}
\end{definition}
A highly useful feature of models of $\T_{\vec{C}}$ is that they are uniquely determined by their height, consequentially the uncountable $\T_{\vec{C}}$-models form a hierarchy below $H(\omega_2)$.
\begin{theorem}
Let $\vec{C}$ be a ladder system in $M$, assumed to be a transitive model of $T_{\vec{C}}$.
Then, $M$ is the unique model of $\T_{\vec{C}}$ of height $M \cap Ord$.
\end{theorem}
\begin{proof}
Assume that $M$ and $M'$ are transitive, $M \cap Ord = M' \cap Ord$, $\vec{C} \in M \cap M'$, which implies that $M$ and $M'$ have the right $\omega_1$, and both $M$ and $M'$ are $\T_{\vec{C}}$ -models.  We work towards a contradiction, so assume that  $X \in M$ yet $X \notin M'$. As every set in $M$ has size at most $\aleph_1$ we can assume that $X \subset \omega_1$, hence there is a real $r_X \in M$ which codes $X$ with the help of the the almost disjoint family $F_{\vec{C}}$. Now $r_X$ is itself coded by a triple of ordinals $(\beta, \gamma, \delta) \in M$, thus there is a reflecting sequence $(N_{\xi} \, : \, \xi < \omega_1) \in M$ witnessing that $r_X$ is determined by $(\beta, \gamma, \delta)$. As $M \cap Ord = M' \cap Ord$, $(\beta, \gamma, \delta)$ is in $M'$ as well, and there is a reflecting sequence $(P_{\xi} \, : \, \xi < \omega_1) \in M'$ which witnesses that $(\beta, \gamma, \delta)$ is stabilized in $M'$. The set $C:= \{ \xi < \omega_1 \, : \, P_{\xi} = N_{\xi} \}$ is a club on $\omega_1$, hence if $\eta$ is a limit point of $C$, the reflecting sequence $(P_{\xi} \, : \, \xi < \omega_1) \in M'$ will stabilize at $\eta$ and compute $r_X$, hence $X$ is an element of $M'$ which is a contradiction.
\end{proof}
It is a fact that under $\BPFA$ and $\omega_1=\omega_1^L$, every subset of $\omega_1$ is coded by a real relative to the canonical almost disjoint family of $L$-reals $D$. In particular we will have reals $r_M$ which code the $\T_{\vec{C}}$-model $M$ in exactly the same way as in Lemma \ref{realscodeTmodels}. We immediately obtain:
\begin{lemma}
Assume that $\BPFA$ is true and $\aleph_1=\aleph_1^L$. Then for every $\T_{\vec{C}}$-model $M$, there is a real $r_M$ such that
\begin{align*}
(\ast) \quad \forall N (|N| = \aleph_0& \, \land N \text{ is transitive and }  \omega_1^N=(\omega_1^{L})^N  \land r_M \in N \\& \rightarrow
N \models  ``L[r_M] \models \text{ $(r_M)$ codes a $\T_{\vec{C}}$-model}"
\end{align*}
If $N$ is uncountable and transitive,  and $r_M$ is as above, then
\begin{align*}
N \models ``\text{The model decoded out of $r_M$ is a $\T_{\vec{C}}$-model.}"
\end{align*}

\end{lemma}

We aim now to argue for the $\Sigma^1_3$-uniformization property. Let $\varphi(v_0,v_1)$ be an arbitrary $\Sigma^1_3$-formula and assume that $x \in 2^{\omega}$ is such that there is a $y$ such that $\varphi(x,y)$ holds. We define the uniformizing function $f$ as in the last section:
\begin{align*}
f(x)=y \Leftrightarrow \exists M &(M \text{ transitive } \land \omega_1 \in M \land M \models \T_{\vec{C}} \land \\& M \models \text{“}\varphi(x,y) \land y \text{ is } < \text{-minimal}" \land \\& M \models \nexists N ( N \models T_{\vec{C}} \land \varphi(x,y)).
\end{align*}
By Shoenfield absoluteness, it is clear that $f$ does uniformize every $\Sigma^1_3$-subset of the plane. The above Lemma allows us to localize the above definition to obtain a $\Sigma^1_3$-definable uniformizing function.

\begin{lemma}
Assume $\BPFA$ is true and $\aleph_1=\aleph_1^L$. Let $\vec{C}$ be the canonical ladder system on $\omega_1^L$. Then the $\Sigma^1_3$-uniformization property is true. Any $\Sigma^1_3$-formula $\varphi(v_0,v_1)$ can be uniformized by the function $f_{\varphi}$ which is defined by the following $\Sigma^1_3$-formula
\begin{align*}
f_{\varphi}(x)=y \Leftrightarrow \exists r \forall N &(N \text{ countable and transitive } \land \omega_1^N= (\omega_1^L)^N \land r \in N \rightarrow  \\& N \models r\text{ codes the least } \T_{\vec{C}}\text{-model } M \text{ which sees that } \\& \exists y' \varphi(x,y') \text{ is true and } y \text{ is the least such witness in } M).
\end{align*}

\end{lemma}
Its proof is entirely analogous to the proof of Lemma \ref{uniformization}, so we skip it.

Going from lightface to boldface causes no difficulties at all and we immediately obtain the next result.

\begin{corollary}
Assume $\BPFA$ and that $\omega_1$ is accessible to some real $r$. Then every $\Sigma^1_3(x)$ relation in the plane can be uniformized by a $\Sigma^1_3(r,x)$-function.

\end{corollary}

\section{Some questions}
We end with a couple of open questions whose answers would require different methods than the one introduced here. 
\begin{question}
Is the $\Sigma^1_n$-uniformization property for $n > 3$ and $\MA$ or $\BPFA$ consistent?
\end{question}

\begin{question}
Is the $\Sigma^1_3$-uniformization property and a large continuum $(> \aleph_2)$  consistent?
\end{question}
Finally, it is natural to ask whether it is possible to have  a  global behaviour of the uniformization property in the presence of forcing axioms.
\begin{question}
Is there a model of $\MA$ for which the uniformization property simultaneously holds true for $\Sigma^1_n$, $n \ge 3$?
\end{question}

\end{document}